\documentclass[journal]{IEEEtran}
\usepackage{amssymb}
\usepackage{hyperref}
\usepackage[official]{eurosym}

\usepackage{graphicx}
\usepackage[utf8]{inputenc}
\usepackage[dvipsnames]{xcolor}
\usepackage{lineno,hyperref}
\usepackage{amsmath}
\usepackage{amsfonts} 
\usepackage{longtable}
\usepackage[official]{eurosym}
\usepackage[dvipsnames]{xcolor}
\usepackage{multicol}
\usepackage{graphicx}
\usepackage{caption}
\usepackage{subcaption}
\usepackage{amsthm}

\usepackage{todonotes}
\setuptodonotes{size=\tiny}

\newtheorem*{theorem*}{Theorem}

\theoremstyle{plain}

\ifCLASSINFOpdf
\else
\fi
\allowdisplaybreaks

\begin{document}
%
\title{Time series aggregation for optimization: One-size-fits-all?}
%
%
%

\author{Sonja~Wogrin,~\IEEEmembership{Senior Member,~IEEE}
\thanks{S. Wogrin, Institute of Electricity Economics and Energy Innovation, Graz University of Technology, Austria,
e-mail: wogrin@tugraz.at}
}

\markboth{Journal of \LaTeX\ Class Files,~Vol.~14, No.~8, August~2015}%
{Shell \MakeLowercase{\textit{et al.}}: Bare Demo of IEEEtran.cls for IEEE Journals}

\maketitle

\begin{abstract}
One of the fundamental problems of using optimization models that use different time series as data input, is the trade-off between model accuracy and computational tractability. 
To overcome computational intractability of these full optimization models, the dimension of input data and model size is commonly reduced through time series aggregation (TSA) methods. However, traditional TSA methods often apply a one-size-fits-all approach based on the common belief that the clusters that best approximate the input data also lead to the aggregated model that best approximates the full model, while the metric that really matters –the resulting output error in optimization results – is not well addressed. 
In this paper, we plan to challenge this belief and show that output-error based TSA methods with theoretical underpinnings have unprecedented potential of computational efficiency and accuracy.


\end{abstract}


\IEEEpeerreviewmaketitle

\section{Introduction}
\label{sec:Introduction}



\IEEEPARstart{T}{he} traditional and vast majority of time series aggregation (TSA) frameworks focus on best approximating the original data (i.e. to reduce difference between cluster centroids and actual data, which we will refer to as the \textit{input error}) with aggregated or clustered data, completely separating the realm of data from the realm of optimization. Such traditional \textit{a-priori }methods are based on the common belief that the clusters that best approximate the data also lead to the aggregated model that best approximates the full model (i.e. minimize the \textit{output error}, the difference between full and aggregated model results), which is not necessarily true, as we will show through an illustrative case. Examples of such a-priori methods and applications, including k-medoids \cite{teichgraeber2019clustering} or k-means, can be found in a recent literature review \cite{hoffmann2020review}. 
Some a-priori TSA methods keep additional information about the original time series that are important for optimization model results, such as adding periods with extreme events, e.g. \cite{dominguez2011selection,de2013optimal,munoz2015endogenous}. While this might improve model outcomes, the choice of extreme days is still taken with respect to input data only. As pointed out by \cite{scott2019clustering}, the correct extreme periods cannot be known in advance because they depend on endogenous optimization outcomes, which leads us to \textit{a-posteriori}\footnote{A-posteriori methods employ preliminary optimizations to improve the aggregation process.} methods. Some examples of a-posteriori methods include \cite{postges2019time, sun2019data,li2022representative}; however, they either contain some kind of heuristic components or are tailored to toy problems.


In this paper we show that first, a-priori TSA methods are not a one-size-fits-all solution when used for optimization models, and that they should ultimately be replaced by a-posteriori methods. To that purpose, we first define full and aggregated optimization models and apply a traditional k-means clustering technique to an illustrative example in section \ref{sec:FullAggregated}.
And second, that when a-posteriori methods are based on theoretical underpinnings - as the basis-oriented method proposed in section \ref{sec:BasisTSA} - they outperform a-priori methods by orders of magnitude. Section \ref{sec:conclusion} concludes the paper.


\section{Full and aggregated optimization models}
\label{sec:FullAggregated}

We consider the following generic formulation of a full (left) and its corresponding aggregated (right) optimization problem:

\begin{minipage}{0.15\textwidth}
  \begin{subequations}\label{eqn:FullModel}
  \begin{align*}
      min    f(x,TS)        \\
      \text{s.t. } g(x,TS) \leq 0 
    \end{align*}
    \end{subequations}
\end{minipage}
\hspace{1.5cm}
\begin{minipage}{0.15\textwidth}
    \begin{subequations}
    \begin{align*}
      min    \overline{f}(\overline{x},\overline{TS}) \\
      \text{s.t. } \overline{g}(\overline{x},\overline{TS}) \leq 0, 
    \end{align*}
    \end{subequations}
\end{minipage}
\vspace{0.3cm}

where $x$ are the decision variables, $TS$ represent the original time series used as data, and $f$ and $g$ are the objective function and constraints. The number of variables is proportional to the cardinality of the time series, $|x| \sim |TS|$, and hence there is a large amount of variables $x$ and a large number of constraints $g$ in the full problem, which often leads to computational complexity and intractability. Through a TSA process often obtain through clustering algorithms, the original TS are transformed into the aggregated $\overline{TS}$, where $|\overline{TS}| \ll |TS|$. Correspondingly $|\overline{TS}| \ll |x|$, and so is the number of constraints, which leads to a significant reduction in computational burden of the aggregated optimization model with respect to the full optimization model. In general, there is no guarantee with respect to the quality of aggregated versus full model results. 

\subsection{Economic dispatch problem}
\label{subsec:ED}

The full economic dispatch (ED) optimization problem minimizes overall power system cost over a time horizon of hours $h$ by determining the optimal production of generating units g, each of which has their corresponding variable operating costs $C_g$, and upper and lower bounds while meeting system demand $D_{g,h}$ at each hour\footnote{If it is not possible to meet demand with existing generators, then there is non-supplied energy at a high cost. We have not modeled this explicitly for simplicity. But a fictitious generator with high operating cost and infinite upper bound can represent non-supplied energy in the presented formulation.}. Lower bound $\underline{P}_{g}$ depends on technical characteristics of the generator itself, but upper bound $\overline{P}_{g,h}$ also depends on the temporal index. Indeed $\overline{P}_{g,h}$ can be obtained as the product of the installed generator capacity $\underline{P}_{g}$ multiplied by its capacity factor $\underline{CF}_{g,h}$\footnote{For thermal generators, the capacity factor is 1, but for generators that belong to variable renewable energy sources, e.g. wind and solar, such a capacity factor depends on solar irradiation or wind speeds, which vary over time.}. The TS of this problem are system demand and capacity factors.
In an aggregated ED problem, we consider only $r$ representative hours, each of which has a weight $W$ and $|r| \ll |h|$. Both the cardinality of $|r|$, the corresponding weights $W_r$ and aggregated data $D_r$ and $\overline{P}_{g,r}$ most likely stem from a data aggregation/clustering procedure. A stylized formulation of the full (left) and aggregated (right) ED is given below:

\begin{minipage}{0.15\textwidth}
  \begin{subequations}\label{eqn:ED_Full}
  \begin{align*}
      min    \sum_{g,h} C_g p_{g,h} \\
  \text{s.t. } \sum_{g} p_{g,h} = D_h \quad \forall h  \\
  \underline{P}_{g} \leq p_{g,h} \leq \overline{P}_{g,h}  \quad \forall g,h \label{eqn:ED_FullBounds}
    \end{align*}
    \end{subequations}
\end{minipage}
\hspace{0.5cm}
\begin{minipage}{0.15\textwidth}
    \begin{subequations}
    \begin{align*}
      min    \sum_{g,r} C_g p_{g,r} W_r \\
  \text{s.t. } \sum_{g} p_{g,r} = D_r \quad \forall r  \\
  \underline{P}_{g} \leq p_{g,r} \leq \overline{P}_{g,r} \quad \forall g,r 
    \end{align*}
    \end{subequations}
\end{minipage}



\subsection{Economic dispatch and k-means clustering}
\label{subsec:Simple Example}

We consider a numerical example of the ED with one wind unit and one thermal unit over the time horizon of one year ($h=1,\cdots,8760$). The model data are hourly time series of demand and wind production factors (the latter affect $\overline{P}_{g,h}$). To obtain aggregated data for the aggregated ED model, we use the probably most frequently used input-error-based a-priori TSA method, i.e., k-means clustering. In order to run k-means, the user has to specify the total desired number of clusters, which can range from 1 to 8760. However, the choice of the number of clusters is often done on a trial and error basis. In this example, we choose 3 clusters - this seemingly small and arbitrary number will become relevant later on.  

When applying k-means on this data and demanding 3 clusters ($|r|=3,|h|=8760$), we obtain TSA results as shown in Figure \ref{fig:KmeansTSA}, where each wind-demand TS pair is plotted as an x and color-coded depending on the cluster it has been assigned to. The mean squared error (MSE) between the original hourly TS and the cluster centroids is 0.0167, which is the minimum error in the input space that can be obtained for 3 clusters. The 3 cluster centroids for demand and wind capacity factors, i.e. $D_r$ and $\overline{P}_{g,r}$), as well as the corresponding cluster weights $W_r$ - also a result of the TSA procedure - are then used as data input in the aggregated economic dispatch model. Running the aggregated economic dispatch model with 3 weighted representative periods using k-means clustered data, yields an overall error of 91\%\footnote{Calculated as the difference between the objective function value of the full model, and the objective function value of the aggregated model.} in total system costs between the full and the aggregated model results. This also shows how inefficient an a-priori TSA method can be when used in aggregated optimization models. Note that in order to achieve a 0\% output error using k-means, all 8760 clusters would have to be used.


\section{Basis-oriented time series aggregation FOR aggregated optimization models}
\label{sec:BasisTSA}

When solving aggregated optimization models, what really matters is how well aggregated model results approximate full model outputs, i.e., the output error, and not the input error. Therefore, in this section we propose an innovative a-posteriori framework of basis-oriented TSA \textbf{for} aggregated optimization models to overcome such inefficiencies. Finally, we apply this new methodology to the same ED problem from before.

\begin{figure}
     \centering
     \begin{subfigure}[b]{0.48\textwidth}
         \centering
         \includegraphics[width=\textwidth]{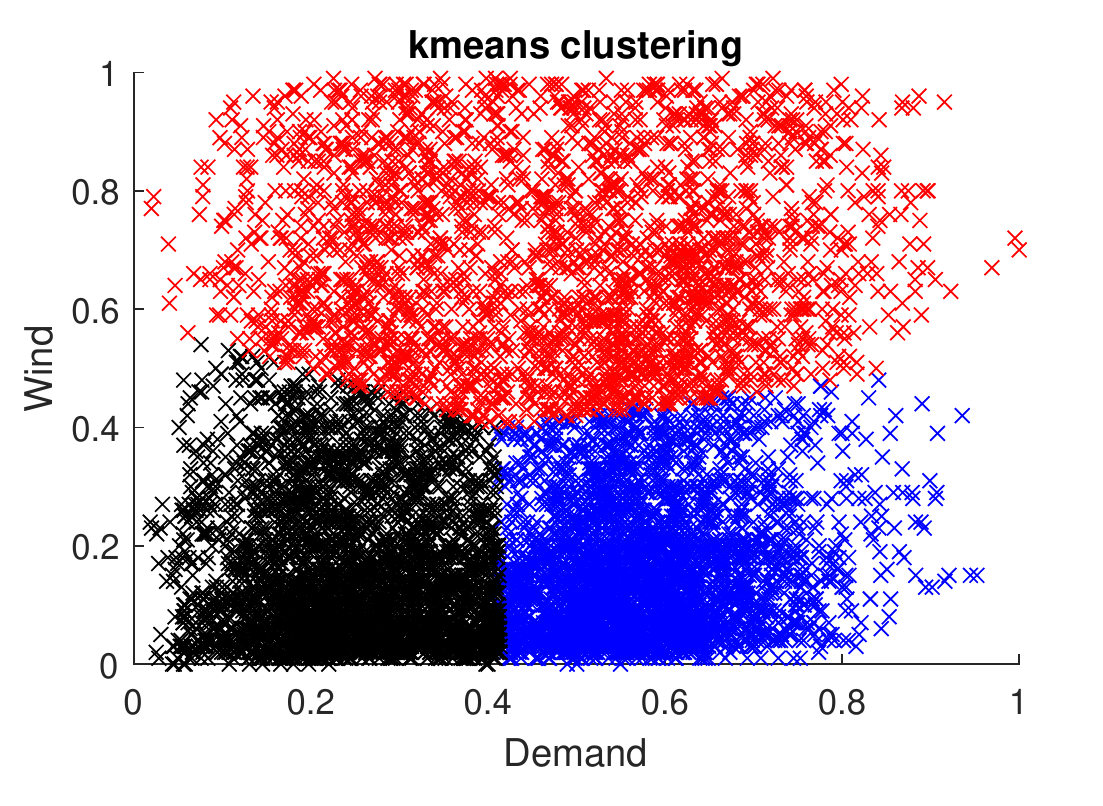}
         \caption{$ $}
         \label{fig:KmeansTSA}
     \end{subfigure}
     \begin{subfigure}[b]{0.48\textwidth}
         \centering
         \includegraphics[width=\textwidth]{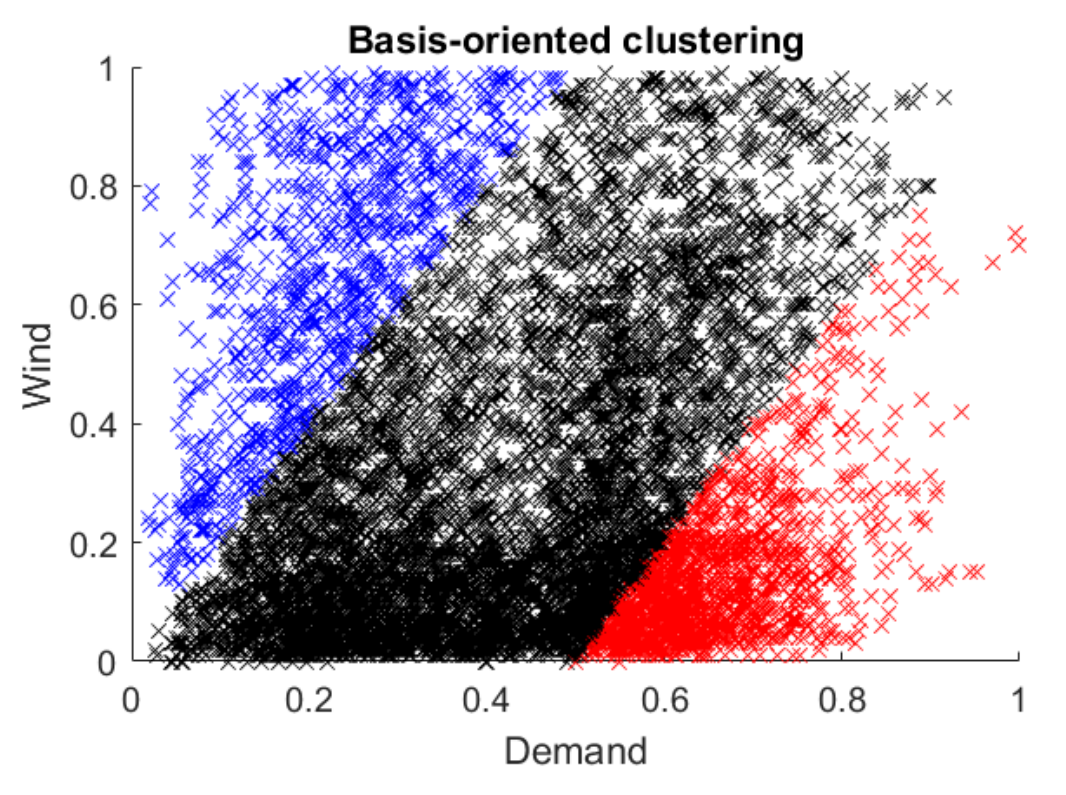}
         \caption{$ $}
         \label{fig:BasisTSA}
     \end{subfigure}
        \caption{K-means (a) and basis-oriented (b) clusters on full time series data.}
        \label{fig:ClusteringExample}
\end{figure}

\subsection{Basis-oriented time series aggregation framework}
\label{subsec:MethBasis}

We consider a linear program (LP) that depends on TS, such as the ED from section \ref{subsec:ED}. For the sake of simplicity, we analyze each time step individually and assume there are no time-period-linking constraints\footnote{If there are no time-period-linking constraints in the optimization problem, then analyzing time steps separately does not incur an error. In any case, the issue of complicating constraints will be a topic of future research.}. In particular, we focus on the optimal solution for this single time step and its corresponding basis $B$ (in the simplex framework). In the remainder of this section, we use the concrete example of the ED to introduce the idea of basis-oriented clustering. We say hour, but it could be a generic time step as well.

Consider two hours with different TS data whose optimal solution of the ED belong to the same basis $B$. The TS data affects the right-hand-side (RHS) vector of the constraints. An expected value (or centroid) of these data, yields an optimal solution that also has the same optimal basis $B$. 

\begin{theorem*} 
For each $i=1,\ldots,I$ consider the following LPs ($E_i$): $\min c^T x$ s.t. $Ax=b_i$, where $x\in \mathbb{R}^n, A\in \mathbb{R}^{mxn}, c\in \mathbb{R}^n,b_i \in \mathbb{R}^m  $ and the LPs only differ in the RHS values $b_i$.
Then, $B$ is also an optimal basis for the problem ($\overline{E}$): $\min c^T x$ s.t. $Ax= \mathbb{E}(b_i)$, where $\mathbb{E}(b_i)= \sum_i \frac{b_i}{I} $.
\end{theorem*}

\begin{proof}
We proof this by contradiction. Assume that $B$ is not the optimal basis for problem ($\overline{E}$): $\min c^T x$ s.t. $Ax= \mathbb{E}(b_i)$. Instead, let $\overline{B}$ ($\neq B$) and $\overline{N}$ be the optimal basis and the non-basis matrices of ($\overline{E}$). Under this assumption, it follows that $c^T_{\overline{B}} x_{\overline{B}} < c^T_B x_B$ for this problem.
In ($\overline{E}$) we obtain that $Ax = \overline{B} x_{\overline{B}} + \overline{N} x_{\overline{N}} = \mathbb{E}(b_i)$. By definition $x_{\overline{N}}=0$, so we further simplify $Ax = \overline{B} x_{\overline{B}}= \mathbb{E}(b_i)$ and it follows that $x_{\overline{B}} = \overline{B}^{-1} ( \mathbb{E}(b_i))$. We now substitute this optimal solution in the objective function of ($\overline{E}$), which yields 
$c^T x = c^T_{\overline{B}} x_{\overline{B}} = c^T_{\overline{B}} \overline{B}^{-1} ( \mathbb{E}(b_i))= c^T_{\overline{B}} \overline{B}^{-1} ( \sum_i \frac{b_i}{I} ) = \frac{1}{I} c^T_{\overline{B}} \overline{B}^{-1} b_1 + \ldots + \frac{1}{I} c^T_{\overline{B}} \overline{B}^{-1}  b_I $. 
Since we know that $B$ is an optimal basis for each problem ($E_i$), we can say that:
$\frac{1}{I} c^T_{\overline{B}} \overline{B}^{-1} b_1 + \ldots + \frac{1}{I} c^T_{\overline{B}} \overline{B}^{-1}  b_I \geq
\frac{1}{I} c^T_B B^{-1} b_1 + \ldots + \frac{1}{I} c^T_B B^{-1}  b_I = c^T_B x_B $. This would imply that $B$ is an optimal basis for ($\overline{E}$), which is a contradiction.
\end{proof}

This result has significant ramifications with respect to clustering TS data: imagine two hours with different data, i.e. $b_1$ and $b_2$, but an identical optimal basis $B$ for the underlying LP, then the full optimization model (with individually represented hours), and the aggregated optimization model where we only have one expected hour (i.e. the cluster centroid or expected value $\frac{b_1+b_2}{2}$), yield the same objective function value, and the same expected results for the variables.
\textbf{Hence, those two hours can be merged without losing any accuracy in the final aggregated model}.
In other words, if hours are aggregated within their basis and represented by the cluster centroid, then the aggregated model results will be exactly the same as the full hourly model results and have zero output error in expectation. This shows: first, that TSA for optimization purposes must be based on the impact of the aggregation on the optimization output error and not on similarity of input data as in traditional methods; second, it introduces basis-oriented TSA as a promising way of clustering \textbf{for} aggregated optimization problems.

\subsection{Economic dispatch and basis-oriented clustering}
\label{subsec:Simple Example basis clustering}

Applying basis-oriented clustering to the above-mentioned ED example shows that there are only 3 different bases. Therefore, we only require 3 clusters. 
In Figure \ref{fig:BasisTSA}, we have color-coded each hour depending on the basis (and corresponding cluster) this hour belongs to: blue (the wind generator is the marginal generator), black (thermal is on the margin), and red (hours with non-supplied energy). The MSE in the input space is 0.0385, which is 2.3 times larger than the MSE obtained by k-means, so the input error is higher with the obtained clusters. However, if we cluster all hours within their basis, then the aggregated optimization results are exact! Solving the aggregated optimization model with only 3 representative periods that have been clustered, accounting for the corresponding bases, yields an error of 0\%. Apart from being theoretically exact, basis-oriented clustering also establishes the maximum number of clusters necessary to obtain exact optimization results, which is nothing current TSA methods can offer.

\section{Conclusion}
\label{sec:conclusion}

The takeaways from this paper are as follows. 
First, a-priori TSA methods are fundamentally flawed when used in/for  optimization models and therefore should be abandoned and replaced by a-posteriori methods. As we have shown by counter-example, the lowest input error (as indicated by MSE) does not translate into the lowest (or even a low) output error when approximating full optimization model results. 
Second, basis-oriented TSA can achieve a tremendous reduction (3 out of 8760 hours) in input data of several orders of magnitude while replicating full model results exactly. 
This confirms that picking clusters intelligently
can outperform traditional one-size-fits-all a-priori TSA methods (such as k-means) even if those use most of the original data. 
Finally, more research is required to develop a theoretical framework on most efficient a-posteriori TSA methods. The basis-oriented TSA proposed here could be a starting point. In future research, we plan to explore basis-oriented TSA further to see if it can be extended to large-scale, realistic problems with, e.g., time-linking and discrete constraints.

\section*{Acknowledgment}

I want to thank D. Cardona-Vasquez, D. Di Tondo, S. Pineda and J.M. Morales for their comments.

\ifCLASSOPTIONcaptionsoff
  \newpage
\fi

\bibliographystyle{elsarticle-num} 
\bibliography{main}




\end{document}